\documentclass[11pt]{amsart}
\usepackage{hyperref}
\usepackage{amsfonts}
\usepackage{amsmath}
\usepackage{amssymb}
\usepackage{amscd}
\usepackage{graphicx}
\usepackage{enumitem}
\usepackage{latexsym}
\usepackage{color}
\usepackage{comment}
\usepackage{epsfig}
\usepackage{amsopn}
\usepackage{xypic}
\usepackage{mathrsfs}
\usepackage{array}
\usepackage[usenames,dvipsnames]{pstricks}
\usepackage{epsfig}
\usepackage{pst-grad} % For gradients
\usepackage{pst-plot} % For axes
\usepackage[space]{grffile} % For spaces in paths
\usepackage{etoolbox} % For spaces in paths
\makeatletter % For spaces in paths
\patchcmd\Gread@eps{\@inputcheck#1 }{\@inputcheck"#1"\relax}{}{}
\makeatother

\hypersetup{pdfpagemode=UseNone}
\usepackage{booktabs}
\usepackage{longtable}
\theoremstyle{plain}
\newtheorem{lemma}{Lemma}[section]
\newtheorem*{theorem*}{Theorem}
\newtheorem*{lemma*}{Lemma}
\newtheorem*{proposition*}{Proposition}
\newtheorem*{conjecture*}{Conjecture}
\newtheorem*{corollary*}{Corollary}
\newtheorem*{problem*}{Problem}
\newtheorem{theorem}[lemma]{Theorem}

\newtheorem{corollary}[lemma]{Corollary}
\newtheorem{proposition}[lemma]{Proposition}

\theoremstyle{definition}
\newtheorem{definition}[lemma]{Definition}

\newtheorem{remark}[lemma]{Remark}

\newcommand{\fto}[1]{\stackrel{#1}{\to}}

\newcommand{\CC}{\mathbb{C}}

\newcommand{\OO}{\mathcal{O}}
\newcommand{\te}{\otimes}

\newcommand{\onto}{\twoheadrightarrow}

\newcommand{\gr}{\mathrm{gr}}

\newcommand{\mumax}{\mu_{\max}}
\newcommand{\mumin}{\mu_{\min}}

\renewcommand{\P}{\mathbb{P}}

\newcommand{\PP}{\mathbb{P}}

\DeclareMathOperator{\ch}{ch}

\DeclareMathOperator{\Hom}{Hom}

\DeclareMathOperator{\Pic}{Pic}

\DeclareMathOperator{\rk}{rk}

\DeclareMathOperator{\coker}{coker}

\DeclareMathOperator{\Ext}{Ext}
\DeclareMathOperator{\ext}{ext}

\DeclareMathOperator{\ev}{ev}

\DeclareMathOperator{\sHom}{\mathcal{H}\kern -.5pt\mathit{om}}
\DeclareMathOperator{\sTor}{\mathcal{T}\kern -1.5pt\mathit{or}}

\DeclareMathOperator{\ecodim}{ecodim}

\DeclareMathOperator{\codim}{codim}

\newcommand{\leqor}{\underset{{\scriptscriptstyle (}-{\scriptscriptstyle )}}{<}}

\begin{document}

\date{\today}
\author[I. Coskun]{Izzet Coskun}
\address{Department of Mathematics, Statistics and CS \\University of Illinois at Chicago, Chicago, IL 60607}
\email{icoskun@uic.edu}

\author[J. Huizenga]{Jack Huizenga}
\address{Department of Mathematics, The Pennsylvania State University, University Park, PA 16802}
\email{huizenga@psu.edu}

\author[N. Raha]{Neelarnab Raha}
\address{Department of Mathematics, The Pennsylvania State University, University Park, PA 16802}
\email{neelraha@psu.edu}

\thanks{During the preparation of this article the first author was partially supported by the NSF grant DMS 2200684
and the second and third authors were partially supported by NSF FRG grant DMS 1664303.  The second author is also grateful for support from the Travel Support for Mathematicians program of the Simons Foundation}

%\title{Bundles with many sections on $\P^2$ and higher rank Brill-Noether theory}

\title{Brill-Noether theory on $\P^2$ for bundles with many sections}

\begin{abstract}
The Brill-Noether theory of curves plays a fundamental role in the theory of curves and their moduli and has been intensively studied since the 19th century. In contrast, Brill-Noether theory for higher dimensional varieties is less understood.  It is hard to determine when Brill-Noether loci are nonempty and these loci can be reducible and of larger than the expected dimension. 

Let $E$ be a semistable sheaf on $\PP^2$. In this paper, we give an upper bound  $\beta_{r, \mu}$ for $h^0(E)$ in terms of the rank $r$ and the slope $\mu$ of $E$. We show that the bound is achieved precisely when $E$ is a twist of a Steiner bundle.   We classify the  sheaves $E$ such that   
$h^0(E)$ is within $\lfloor \mu(E) \rfloor + 1$ of $\beta_{r, \mu}$. We determine the nonemptiness, irreducibility and dimension of  the Brill-Noether loci in the moduli spaces of sheaves on $\PP^2$ with $h^0(E)$ in this range. When they are nonempty, these Brill-Noether loci are irreducible  though almost always of larger than the expected dimension.
\end{abstract}
\maketitle

\setcounter{tocdepth}{1}
\tableofcontents

\section{Introduction}

Brill-Noether Theory has been the cornerstone of the study of curves and their moduli since the 19th century. A major triumph of modern algebraic geometry in the 1980s is a description of the cohomology jumping loci $$W_d^r := \{ L \in \Pic^d(C)| h^0(L) \geq r+1\}$$ for a general curve $C$ (see \cite{GriffithsHarris, FultonLazarsfeld, Gieseker3, EisenbudHarris}). Similarly, inspired by conjectures of Mukai \cite{Mukai}, Bertram and Feinberg \cite{BertramFeinberg}, the Brill-Noether loci in moduli spaces of  higher rank vector bundles on a general curve $C$ have been intensively studied (see \cite{GM, Newstead} for surveys). 

In contrast, the Brill-Noether loci for moduli spaces of sheaves on surfaces is less understood. Unlike in the case of curves, where the Euler characteristic determines the cohomology of the general stable sheaf of rank $r$ and degree $d$, on a surface the cohomology of the general stable sheaf can be hard to determine and has been the subject of intense study in recent years (see \cite{GottscheHirschowitz, CoskunHuizengaWBN, CoskunHuizengaBN, CoskunHuizengaNuer, CoskunNuerYoshioka, CoskunNuerYoshioka2, LevineZhang}). On  $\PP^2$, by a celebrated theorem of G\"{o}ttsche and Hirschowitz \cite{GottscheHirschowitz}, the general stable bundle has at most one nonzero cohomology group. Hence, the cohomology is determined by the Euler characteristic and the slope.  Therefore, it makes sense to consider the cohomology jumping loci.

\subsection{Bounding $h^0(E)$} Let $E$ be a semistable sheaf on $\PP^2$. If $\mu(E) <0$, then $E$ cannot have any global sections, so we may assume $\mu(E)\geq 0$. Our first main theorem gives a bound on $h^0(E)$. This theorem is an analogue of Clifford's Theorem for higher rank semistable sheaves on $\PP^2$.

\begin{theorem}[Corollary \ref{cor-sectionBound}]
Let $E$ be a semistable sheaf of rank $r$ and slope $\mu \geq 0$ on $\PP^2$. Let $\alpha =  \lfloor \mu \rfloor + 1$.  Then 
$$h^0(E) \leq  r \alpha \left(\mu- \frac{\alpha}{2} + \frac{3}{2}\right).$$
\end{theorem}

\begin{remark}
More generally, Theorem \ref{thm-sectionBound} and Corollary \ref{cor-sectionBound} will give bounds on $h^0(E)$ when $E$ is a torsion-free coherent sheaf whose Harder-Narasimhan factors satisfy appropriate bounds.
\end{remark}

\begin{remark}
For a stable sheaf $E$ with $\mu(E)>0$ on $\PP^2$, Gould, Lee and Liu \cite{GouldLiu} give the upper bound 
$$h^0(E) \leq \max\left(\frac{1}{2}c_1(E)^2 + \frac{3}{2} c_1(E)+1, r \right).$$ Their upper bound is sharp when it is computed by $r$. Otherwise, the leading term of our upper bound is stronger by a factor of $r$. We will shortly see large intervals of $\mu$ where our upper bound is attained by stable sheaves.   Recently, Costa, Mac\'ias Tarr\'{i}o and  Roa-Leguizam\'{o}n \cite{CMRL} gave a bound on $h^0$ of semistable sheaves on any surface with $K_X \cdot H \leq 0$. Our bounds also improve theirs in the case $X= \PP^2$.
\end{remark}

 In light of the theorem, we define $$\alpha= \alpha_{\mu} := \lfloor \mu \rfloor + 1 \quad \mbox{and} \quad \beta= \beta_{r, \mu} : = r \alpha \left(\mu -  \frac{\alpha}{2} + \frac{3}{2} \right),$$ where $\beta_{r, \mu}$ is the maximum possible $h^0(E)$ for a semistable sheaf $E$ of rank $r$ and slope $\mu$.

 \begin{definition}
 We say that a semistable sheaf $E$ of rank $r$ and slope $\mu\geq 0$ has {\em deficiency} $\delta$ if $$h^0(E) =\beta_{r, \mu} - \delta.$$
 \end{definition}
 
 Let $M(r, \mu, \chi)$ denote the moduli space of Gieseker semistable sheaves on $\PP^2$ of rank $r$, slope $\mu$ and Euler characteristic $\chi$.
 Note that Dr\'ezet and Le Potier have completely classified when the moduli spaces $M(r,\mu,\chi)$ are nonempty \cite{DrezetLePotier}.  Let $B^k(r, \mu, \chi) \subset M(r, \mu, \chi)$ denote the Brill-Noether locus defined as the closure of the locus of stable sheaves that have at least $k$ linearly independent sections. The main goal of the rest of the paper is to classify sheaves with small deficiency and study the corresponding Brill-Noether loci.
 
\subsection{Sheaves with $\delta=0$} Let $E$ be a semistable sheaf of rank $r$, slope $\mu\geq 0$, and deficiency $\delta=0$. Then $E$ has the maximum possible $h^0$ with the given rank and slope. We will see that then $E$ is a \emph{twisted Steiner bundle} of the form$$0 \to \OO_{\PP^2}(\alpha-2)^a \fto{M} \OO_{\PP^2}(\alpha-1)^{a+r} \to E \to 0$$ for some integer $a\leq r$.  Any such sheaf has $h^1(E) = 0$ and is globally generated.  These sheaves were  studied by Dr\'ezet on $\PP^2$ \cite{Drezet} and  generalized to any $\PP^n$ by Dolgachev and Kapranov \cite{DolgachevKapranov}. Their stability and properties have been studied extensively (see \cite{Brambilla, Brambilla2, CoskunHuizengaSmith, Huizenga}).  When these sheaves are semistable, they constitute entire moduli spaces of semistable sheaves.  Therefore, Brill-Noether loci $B^\beta (r,\mu,\chi)$ are either empty or the entire moduli space, and it is easy to tell which possibility holds.

\subsection{Sheaves with   deficiency less than $ \alpha$}

When the deficiency $\delta $ is less than  $\alpha$, we will not describe the sheaves quite so explicitly.  However, the following result generalizes one key aspect of the deficiency $0$  case. 

\begin{theorem}[Theorem \ref{thm-h1}]\label{thm-Introh1}
Let $E$ be a semistable sheaf on $\PP^2$ of rank $r$, slope $\mu \geq 0$ and deficiency $\delta$. 
If $\delta < \alpha$, then $H^1(E)=0$ and $E$ is globally generated.
\end{theorem}

Hence the Brill-Noether loci $B^{\beta-\delta}(r,\mu,\chi)$ are either empty or the entire moduli space, and we can determine which possibility holds.    More precisely, we have the following.

\begin{corollary}\label{cor-introBNdla}  Suppose $\delta < \alpha$, and let $M(r,\mu,\chi)$ be a nonempty moduli space of sheaves on $\P^2$. The Brill-Noether locus $B^{\beta- \delta}(r,\mu,\chi)$ is empty unless $\chi \geq \beta - \delta$. Furthermore, every  $E \in M(r, \mu, \beta- \delta)$ has $h^0(E) = \beta - \delta$. Hence, $B^{\beta - \delta}(r, \mu, \beta- \delta) = M(r, \mu, \beta- \delta)$ and $B^{\beta - \delta+1}(r, \mu, \beta- \delta)= \emptyset$.
 \end{corollary}
 
  The proof of Theorem \ref{thm-Introh1} occupies the first half of the paper.

\subsection{Sheaves with deficiency $\alpha$}

In the second half of the paper, we give an in depth study of the Brill-Noether loci $B^{\beta-\alpha}(r,\mu,\chi)$.  In light of Corollary \ref{cor-introBNdla},  these  are the first interesting Brill-Noether loci which arise when we decrease the number of sections from $\beta$.  We study the basic questions of nonemptiness, irreducibility, and dimension for these loci, and the answers are quite delicate.

First, we need a generalization of Theorem \ref{thm-Introh1} to the case of $\delta = \alpha$.

\begin{theorem}[Theorem \ref{thm-defAlpha}]
Let $E$ be a $\mu$-semistable sheaf on $\PP^2$ of rank $r\geq 2$ and slope $\mu$ having $h^0(E) = \beta-\alpha$.    Then either\begin{enumerate}
\item $H^1(E) = 0$ and $E$ is globally generated, or 

\item There is a twisted Steiner bundle $S$ of the form $$0\to \OO(\alpha-2)^{a-1}\to \OO(\alpha-1)^{a+r-1}\to S\to 0$$ such that $E$ fits as an extension $$0\to S\to E\to \OO_L(-b)\to 0$$ for some line $L$, where $b=1+h^1(E)$.  In this case, $E$ fails to be globally generated in codimension 1.
 \end{enumerate}
\end{theorem}

The first case only arises when $\chi = \beta - \alpha$.  In that case, $B^{\beta-\alpha}(r,\mu,\chi)$ is the entire moduli space.  

On the other hand, when $\chi < \beta - \alpha$ then the situation is much more interesting.  Then the locus $B^{\beta-\alpha}(r,\mu,\chi)$ is nonempty if and only if there are \emph{semistable} extensions of the form   $$0\to S\to E\to \OO_L(-b)\to 0.$$  At the end of the paper we analyze the semistability of these extensions, and get the following consequences for  Brill-Noether loci.

\begin{theorem}[Theorem \ref{thm-existence}, Corollary \ref{cor-irr} and Proposition \ref{prop-dim}]\label{thm-Introexistence}Let $\mu > 0$ be a non-integer, and express it as $\mu = \alpha-1 + \frac{a}{r}$. Assume that $\chi < \beta - \alpha$.  We let $\varphi = (1+\sqrt{5})/2$ be the golden ratio.
\begin{enumerate}
\item If $\frac{a}{r} > \varphi-1$, then $B^{\beta - \alpha}(r,\mu,\chi)$ is nonempty.
\item  Suppose $\frac{a}{r} < \varphi -1.$  Consider the sequence of ratios of Fibonacci numbers $$\frac{0}{1},\frac{1}{2},\frac{3}{5},\frac{8}{13},\frac{21}{34},\frac{55}{89},\ldots$$ which monotonically converges to $\varphi-1$.  Let $\rho$ be the smallest such ratio with $\frac{a-1}{r} \leq \rho$.

\begin{enumerate}
\item If $\rho \geq \frac{a}{r}$, then $B^{\beta-\alpha}(r,\mu,\chi)$ is empty.

\item Otherwise, $B^{\beta-\alpha}(r,\mu,\chi)$ is nonempty if $\chi \ll 0$.
\end{enumerate}
\end{enumerate}
The Brill-Noether locus $B^{\beta-\alpha}(r,\mu,\chi)$ is irreducible if it is nonempty.  If furthermore $\frac{a-1}{r}> \varphi-1$, then the codimension of $B^{\beta-\alpha}(r,\mu,\chi)$ is $$ r(\beta-\alpha-\chi)+c_1-1.$$ This is strictly smaller than the expected codimension $(\beta-\alpha)(\beta-\alpha-\chi)$ except in the case where $\alpha = 1$ and $\chi = \beta-\alpha-1$.
\end{theorem}

\begin{remark}
Finding out the precise value of $\chi$ that works in Theorem \ref{thm-Introexistence} (2b) is challenging.
In general, when $S$ is not stable, a general extension of the form $$0 \to S \to E \to \OO_L(-b) \to 0$$ may not be stable unless $b\gg0$. For example, there may be no stable sheaves with the Chern character of $E$.  It would be interesting to know whether there exist stable sheaves $E$ arising as such extensions as soon as the Chern character of $E$ is the Chern character of a stable sheaf.
\end{remark}

\subsection*{Organization of the paper} In \S \ref{sec-prelim}, we recall basic facts about moduli spaces of sheaves on $\PP^2$. In \S \ref{sec-bound}, we prove Theorem \ref{thm-sectionBound} which gives an upper bound on $h^0(E)$. In \S \ref{sec-gg}, we prove Theorem \ref{thm-h1} which shows the vanishing of $h^1(E)$ and global generation of $E$ if $h^0(E) > \beta - \alpha$ or if $h^0(E)= \beta - \alpha$ and $E$ is globally generated in codimension 1. In  \S \ref{sec-steiner}, we classify sheaves $E$ that attain the maximum $h^0(E)= \beta$.  In \S \ref{sec-alpha}, we study the Brill-Noether loci $B^{\beta - \alpha}$ and discuss their nonemptiness, irreducibility and dimension.

\subsection*{Acknowledgments} We would like to thank Ben Gould, John Kopper, Woohyung Lee, Yeqin Liu, and Howard Nuer for stimulating discussions about Brill-Noether theory of sheaves on surfaces.

\section{Preliminaries}\label{sec-prelim}
In this section, we collect basic facts and definitions regarding stability and moduli spaces of sheaves on $\PP^2$. We refer the reader to \cite{CoskunHuizengaGokova, HuybrechtsLehn, LePotier} for proofs and more details.

\subsection{Global generation} 
Let $E$ be a torsion-free coherent sheaf and let $$\ev: H^0(E) \otimes \OO_{\PP^2} \to E$$ be the canonical evaluation morphism. 
\begin{enumerate}
\item  $E$ is {\em globally generated} if $\ev$ is surjective. 
\item $E$ is {\em generically globally generated} if the cokernel of $\ev$ is torsion, or equivalently, if the sections of $E$ span a general fiber of $E$.
\item $E$ is {\em globally generated in codimension 1} if the cokernel of $\ev$ is supported on a subscheme of dimension at most 0.
\end{enumerate}
The following implications are immediate:
$$\mbox{Global generation} \Rightarrow \mbox{Global generation in codimension 1} \Rightarrow \mbox{Generic global generation.}$$

\subsection{Slope stability} For a torsion-free coherent sheaf $E$ on a polarized variety $(X,H)$ of dimension $n$, the slope $\mu(E)=\mu_H(E)$ is defined by 
$$\mu(E) = \frac{c_1(E)\cdot H ^{n-1} }{\rk(E)H^n} .$$ The sheaf $E$ is $\mu$-(semi)stable if for all proper subsheaves $F \subset E$ of smaller rank, we have $$\mu(F) \leqor \mu(E).$$
Every torsion-free sheaf $E$ admits a unique Harder-Narasimhan filtration
$$0=E_0 \subset E_1 \subset \cdots \subset E_n = E$$ such that 
$F_i := E_i/E_{i-1}$ are $\mu$-semistable and $$\mu(F_1) > \cdots > \mu(F_n).$$ We define $\mu_{\max}(E) := \mu(F_1)$ and $\mu_{\min}(E) := \mu(F_n)$. Moreover, $\mu$-semistable sheaves admit a Jordan-H\"{o}lder filtration with stable quotients.

\subsection{Gieseker stability} Let $E$ be a pure $d$-dimensional sheaf on a polarized variety $(X,H)$. Define the {\em Hilbert polynomial} $P_E(m)$ and the {\em reduced Hilbert polynomial} $p_E(m)$ by  $$P_E(m) := \chi(E(mH)) = a_d \frac{m^d}{d!} + \mbox{l.o.t}, \quad p_E(m) := \frac{P_E(m)}{a_d}.$$
The sheaf $E$ is Gieseker (semi)stable if for every proper subsheaf $F$, we have $$p_F(m) \leqor p_E(m)$$ for $m\gg0$. When we discuss (semi)stability we will always mean Gieseker (semi)stability.  For torsion-free sheaves, we have the implications
$$\mu\mbox{-stable} \Rightarrow \mbox{stable} \Rightarrow \mbox{semistable} \Rightarrow \mu\mbox{-semistable}.$$ 
 Every pure coherent sheaf admits a unique Harder-Narasimhan filtration for Gieseker semistability. Moreover, every semistable sheaf admits a Jordan-H\"{o}lder filtration with stable quotients. While the filtration is not unique, the associated graded object is unique up to isomorphism. Two semistable sheaves are called {\em S-equivalent} if their associated graded objects for the Jordan-H\"{o}lder filtration are isomorphic.

Whenever we talk about stability on $\P^2$ we always use the polarization $H=\OO_{\P^2}(1)$. Also, we simply write $c_1(E)$ for $c_1(E)\cdot \OO_{\P^2}(1)$.

\subsection{Moduli spaces} Given a polarized variety $(X,H)$,  Gieseker \cite{Gieseker} and Maruyama \cite{Maruyama}  constructed projective moduli spaces of sheaves $M_{X,H} ({\bf v})$ parameterizing S-equivalence classes of semistable sheaves on $X$ with a fixed Chern character ${\bf v}$. 

When $X= \PP^2$, these moduli spaces have been studied in great detail. Dr\'ezet and Le Potier \cite{DrezetLePotier} have classified Chern characters of stable sheaves on $\PP^2$. 
For our purposes, it will be most convenient to record the information of the Chern character by the rank $r$, the slope $\mu$ and the Euler characteristic $\chi$.  We will write $M(r, \mu, \chi)$ for $M_{\PP^2, \OO(1)}(r, \mu, \chi)$.
When nonempty, the moduli space $M(r, \mu, \chi)$ is a normal, factorial,  irreducible projective variety of the expected dimension $1- \chi({\bf v}, {\bf v})$. If $M(r, \mu, \chi)$ is positive dimensional, then the stable locus  $M(r, \mu, \chi)^s$ is nonempty and  $M(r, \mu, \chi)$ is smooth along $M(r, \mu, \chi)^s$. 

Let $B^k(r, \mu, \chi) \subset M(r, \mu, \chi)$ denote the Brill-Noether locus defined as the closure of the locus of stable sheaves that have at least $k$ linearly independent sections.  Within $M(r,\mu,\chi)^s$, the locus $B^k(r, \mu, \chi)$ has a natural determinantal scheme structure \cite{CostaMR} (see also \cite[Proposition 2.7]{CoskunHuizengaWoolf}). Consequently, every component of $B^k(r, \mu, \chi)$ has codimension  less than or equal to the expected codimension
$$\ecodim(B^k(r, \mu, \chi)) = k (k-\chi).$$
When $r, \mu$ and $\chi$ are clear from the context, we will drop them from the notation.

\begin{remark} Gould, Lee and Liu study certain Brill-Noether loci in $M(r, \mu, \chi)$. They show that $B^r(r, \mu, \chi)$ is nonempty. Further, they study the case $\mu=\frac{1}{r}$ in detail and find that all the Brill-Noether loci in this case are irreducible and of the expected dimension. In contrast, when $\mu > \frac{1}{2}$ and is not an integer and $\chi \ll 0$, they find examples of reducible Brill-Noether loci with components of different dimensions \cite[Theorem 1.1]{GouldLiu}. \end{remark}

\section{Bound on sections}\label{sec-bound}

In this section, we give an upper bound on $h^0(E)$ under various assumptions on the Harder-Narasimhan filtration of $E$.  For a semistable sheaf $E$, this gives the bound $h^0(E) \leq \beta_{r,\mu}$ from  the introduction.

\begin{theorem}\label{thm-sectionBound}
Let $E$ be a torsion-free sheaf of rank $r$ on $\P^2$ with $\mumax(E) \geq -1$.  Let $\alpha = \lfloor \mumax (E)\rfloor +1$.  Then $$h^0(E) \leq r\alpha\left(\mumax(E)-\frac{1}{2}\alpha+\frac{3}{2}\right).$$
\end{theorem}
\begin{proof}
The proof is by induction on both $\alpha$ and $r$.  When $\alpha=0$, we have $\mumax(E)<0$ and so $E$ has no sections by the semistability of its Harder-Narasimhan factors.  Thus both sides of the inequality are $0$ and the result holds.  Next suppose $r=1$, in which case $E\cong I_Z(k)$ for some $0$-dimensional subscheme $Z\subset \P^2$ and some integer $k\geq-1$. Then $\mumax(E)=\mu(E)=k$, $\alpha=k+1$, and $$h^0(E)\leq h^0(\OO(k))={k+2\choose 2}=r\alpha\left(\mumax(E)-\frac{1}{2}\alpha+\frac{3}{2}\right),$$ and the asserted bound holds. We now suppose both $\alpha>0$ and $r>1$.  Then the right hand side is positive, so we  may assume $E$ has a section.  
 
\emph{Case 1:} suppose that $E$ is generically globally generated.  In this case, we let $L$ be a general line (which therefore does not meet any singularities of $E$) and consider the restriction sequence $$0\to E(-1)\to E\to E|_L\to 0.$$ Since $E$ is generically globally generated, $E|_L$ is globally generated.  Therefore, $h^1(E|_L)=0$ and $h^0(E|_L) = \chi(E|_L) = r+c_1(E).$ By induction on $\alpha$, the theorem holds for $E(-1)$ and \begin{align*}h^0(E) &\leq h^0(E(-1)) + h^0(E|_L)\\
& \leq r(\alpha-1)\left(\mumax(E)-1-\frac{1}{2}(\alpha-1)+\frac{3}{2}\right) + r+c_1(E)\\
& \leq r(\alpha-1)\left(\mumax(E)-\frac{1}{2}\alpha+1\right)+r(\mumax(E)+1)\\
& = r\alpha\left(\mumax(E)-\frac{1}{2}\alpha+\frac{3}{2}\right).
\end{align*}
Thus the theorem holds in this case.

\emph{Case 2:} suppose that $E$ is not generically globally generated.  We consider the canonical evaluation map $H^0(E) \te \OO_{\P^2} \to E$, and let its image be $F\subset E$.  Then $F$ is torsion-free,  globally generated, has rank $r'<r$, and $h^0(E)=h^0(F)$.  Furthermore, we have $\mumax(F)\leq \mumax(E)$ since $F$ is a subsheaf of $E$.  Let $\alpha' = \lfloor \mumax(F)\rfloor+1$, so that $\alpha'\leq \alpha$.  If $\alpha' = \alpha$, it is clear that $$h^0(E)=h^0(F) \leq r'\alpha'\left(\mumax(F) -\frac{1}{2}\alpha'+\frac{3}{2}\right) < r\alpha\left(\mumax(E)-\frac{1}{2}\alpha +\frac{3}{2}\right),$$
and we are done.  Suppose instead that $\alpha'<\alpha$.  Since $\mumax(F) < \alpha'$ and $\mumax(E) \geq \alpha-1$, we estimate
\begin{align*}h^0(E) = h^0(F) &\leq r'\alpha'\left(\mumax(F) -\frac{1}{2}\alpha'+\frac{3}{2}\right)\\
&< r'\alpha'\left(\alpha'-\frac{1}{2}\alpha'+\frac{3}{2}\right)\\
&= r'\left( {\alpha'+2\choose 2} -1\right)\\
&< r {\alpha+1\choose 2}\\
&\leq r\alpha\left(\mumax(E)-\frac{1}{2}\alpha+\frac{3}{2}\right).
\end{align*}
This completes the proof.
\end{proof}

If the sheaf $E$ is not too unstable, then the bound can be improved.
We first fix some notation.

\begin{definition}\label{def-alpha_beta}
For a rank $r$ and slope $\mu\geq -1$, we define \begin{align*}\alpha=\alpha_\mu&:= \lfloor \mu\rfloor+1,\\
\beta=\beta_{r,\mu}&:= r\alpha\left(\mu-\frac{1}{2}\alpha+\frac{3}{2}\right).
\end{align*}
\end{definition}

\begin{corollary}\label{cor-sectionBound}
Let $E$ be a torsion-free sheaf on $\P^2$ with rank $r$ and slope $\mu \geq -1$, and let $\alpha= \alpha_\mu$ and $\beta = \beta_{r,\mu}$.  Suppose that $$\alpha - 2 \leq \mumin(E) \leq\mumax(E) < \alpha.$$   Then $$h^0(E) \leq \beta.$$
\end{corollary}
\begin{proof}
If $\alpha=0$, both sides are $0$. We henceforth assume $\alpha>0$. Consider the Harder-Narasimhan filtration of $E$, with factors $\gr_1,\ldots,\gr_\ell$.  Let $r_i,\mu_i$ be the rank and slope of $\gr_i$, and let $\alpha^i = \alpha_{\mu_i}$ and $\beta^i = \beta_{r_i,\mu_i}$.  By Theorem \ref{thm-sectionBound}, we have $h^0(\gr_i) \leq \beta^i$.  Our assumptions on $\mumin(E)$ and $\mumax(E)$ give that each $\alpha^i$ equals either $\alpha-1$ or $\alpha$.  If $\alpha^i = \alpha-1$, then it is easy to verify that $$\beta^i = r_i (\alpha-1)\left(\mu_i-\frac{1}{2}(\alpha-1)+\frac{3}{2}\right)\leq r_i\alpha\left(\mu_i-\frac{1}{2}\alpha+\frac{3}{2}\right)$$ since $\mu_i \geq \alpha-2$.  If instead $\alpha^i = \alpha$ then we get $\beta^i = r_i\alpha(\mu_i-\frac{1}{2}\alpha+\frac{3}{2})$ by definition.  Then we compute $$h^0(E) \leq \sum_i h^0(\gr_i)\leq \sum_i\beta^i \leq \sum_i r_i \alpha\left(\mu_i-\frac{1}{2}\alpha+\frac{3}{2}\right)= r\alpha\left(\mu-\frac{1}{2}\alpha+\frac{3}{2}\right)=\beta$$ since $r = \sum_i r_i$ and $\mu = \frac{1}{r} \sum_i r_i\mu_i$.
\end{proof}

The main results of this paper concern the structure of semistable sheaves with close to the maximal number $\beta$ of sections. We therefore make the following definition.

\begin{definition}
If a $\mu$-semistable sheaf $E$ with rank $r$ and slope $\mu$ has $h^0(E)=\beta - \delta$, we say that $E$ has \emph{deficiency $\delta$}.  A sheaf with deficiency $0$ is \emph{maximal}.
\end{definition}   

Before proceeding we note a couple obvious facts about the $\beta$ function defined in Definition \ref{def-alpha_beta}.

\begin{lemma}\label{lem-obvious_beta_facts}
\begin{enumerate}
\item For an integer $k\geq -1$, we have $\alpha_k = k+1$ and $$\beta_{r,k}=r{k+2\choose 2} = h^0(\OO_{\P^2}(k)^{\oplus r}).$$
\item If $r\geq 1$ and $0\leq \mu'<\mu$, then $\beta_{r,\mu'}< \beta_{r,\mu}$.

\item  $\beta_{r,\mu-1}= \beta_{r,\mu}-r(\mu+1).$
\end{enumerate}
\end{lemma}
\begin{proof}
(1) is clear.  The proof of (2) is easily adapted from the final part of the proof of Theorem \ref{thm-sectionBound}.  Part (3) is a direct computation.
\end{proof}

\section{Global generation and nonspeciality}\label{sec-gg}

The main goal of this section is to show that semistable sheaves with sufficiently many sections must be nonspecial and have good global generation properties. Our primary result in this direction is the following.
 
 \begin{theorem}\label{thm-h1}
Let $E$ be a $\mu$-semistable sheaf on $\P^2$ of rank $r$, slope $\mu\geq 0$, and deficiency $\delta$, so that $h^0(E) = \beta-\delta$.

\begin{enumerate}
\item If $\delta < \alpha$, then $H^1(E) = 0$ and $E$ is globally generated.

\item If $\delta=\alpha$ and $E$ is globally generated in codimension $1$, then $H^1(E)=0$ and $E$ is actually globally generated.
\end{enumerate}
\end{theorem}

The first part of the theorem has immediate implications for the study of  Brill-Noether loci.

\begin{corollary}\label{cor-BNsmall}
Consider a nonempty moduli space $M = M(r,\mu,\chi)$ of Gieseker semistable sheaves with rank $r\geq 1$, slope $\mu \geq 0$, and Euler characteristic $\chi$.  If $0\leq \delta < \alpha$, then $B^{\beta - \delta}$ is nonempty if and only if $\chi \geq \beta - \delta$.  In this case $B^\chi = M$ and $B^{\chi+1}$ is empty.
\end{corollary}

Essentially, Brill-Noether loci $B^{\beta-\delta}$ don't become interesting until $\delta$ is at least $\alpha$.  The proof of the theorem will follow from a series of smaller results, and will occupy the rest of the section.

\subsection{Global generation}
In this subsection, we give some preliminary results on the global generation of sheaves with small deficiency.

\begin{lemma}\label{lem-ggg}
Let $E$ be a $\mu$-semistable sheaf with rank $r$, slope $\mu\geq0$, and deficiency less than ${\alpha+1\choose 2}$.  That is, $$h^0(E) > \beta - {\alpha+1\choose 2}.$$  Then $E$ is generically globally generated.
\end{lemma}

Examples like $E = \OO(1)\oplus I_{p_1,p_2,p_3}(1)$ for non-collinear points $p_1,p_2,p_3$ show that this result is sharp.

\begin{proof}
Assume $E$ is not generically globally generated.  We let $F\subset E$ be the image of the canonical evaluation mapping $H^0(E) \te \OO_{\P^2}\to E$, so that $h^0(F) = h^0(E)$.  Let $r'$ be the rank of $F$, so that $r'<r$ since $E$ is not generically globally generated.  We let $\mu' = \mumax(F)$, noting that $F$ need not be semistable. We have $\mu'\leq \mu$ by the $\mu$-semistability of $E$.   Let $\alpha'=\alpha_{\mu'}$ and $\beta' = \beta_{r',\mu'}$.

First we discuss the case where $\alpha' = \alpha$.
The main idea of the proof is to show that $$\beta' \leq \beta-{\alpha+1\choose 2},$$ since then by Theorem \ref{thm-sectionBound} $$h^0(F) \leq \beta' \leq \beta - {\alpha+1\choose 2} < h^0(E),$$ and therefore $h^0(F) < h^0(E)$, a contradiction.  We compute \begin{align*}\beta - \beta'&= r\alpha\left(\mu-\frac{1}{2}\alpha+\frac{3}{2}\right)-r'\alpha\left(\mu'-\frac{1}{2}\alpha+\frac{3}{2}\right)\\
 &\geq r\alpha\left(\mu-\frac{1}{2}\alpha+\frac{3}{2}\right) - (r-1)\alpha\left(\mu-\frac{1}{2}\alpha+\frac{3}{2}\right) \\ &= \alpha \left( \mu - \frac{1}{2}\alpha + \frac{3}{2}\right)\\
& \geq \alpha \left(\alpha-1 -\frac{1}{2}\alpha + \frac{3}{2}\right)\\
&= {\alpha+1\choose 2}.
\end{align*}

What we have actually shown is that if $r'<r$, $\mu'\leq\mu$, and $\alpha' = \alpha$, then $$\beta_{r',\mu'} \leq \beta_{r,\mu} -{\alpha+1\choose 2}.$$
Now if instead $\alpha'<\alpha$, we have $\mu'<\alpha'\leq\alpha-1 \leq \mu$ so we can estimate $$\beta_{r',\mu'} < \beta_{r',\alpha-1}\leq
\beta_{r,\mu}-{\alpha+1\choose 2}$$ using Lemma \ref{lem-obvious_beta_facts} (2). This completes the proof.
\end{proof}

\begin{remark}\label{rem-ggg}
By analyzing the proof of Lemma \ref{lem-ggg}, we can get further information in the case where $h^0(E) = \beta - {\alpha+1 \choose 2}$ and $E$ is not generically globally generated.  In the notation of the proof, we must have $\alpha' = \alpha$, $r' = r-1$, $\mu' = \mu$, and $\mu = \alpha-1$.
\end{remark}

If the deficiency is smaller, we get stronger results on global generation.

\begin{lemma}\label{lem-ggcodim1}
Let $E$ be a $\mu$-semistable sheaf with rank $r$ and slope $\mu\geq 0$.
\begin{enumerate}
\item If $E$ has deficiency less than $\alpha$, then $E$ is globally generated in codimension $1$.
\item If $\mu>0$, $E$ has deficiency $\alpha$, and $E$ is not globally generated in codimension $1$, then the locus where $E$ fails to be globally generated contains a unique curve.  It is a line.
\end{enumerate}
\end{lemma}
\begin{proof}
Since $\alpha \leq {\alpha+1\choose 2}$ with equality for $\alpha=1$, Lemma \ref{lem-ggg} shows that $E$ is generically globally generated unless perhaps if we are in case (2) and $\alpha =1$.  However, in that case, Remark \ref{rem-ggg} shows that $\mu=0$, contradicting our assumption.  Therefore, $E$ is generically globally generated in either case.

As in the proof of Lemma \ref{lem-ggg}, we assume that $E$ is not globally generated in codimension $1$ and let $F\subset E$ be the image of the canonical evaluation.  We have an exact sequence $$0\to F\to E\to T\to 0$$ where $T$ is a torsion sheaf with $1$-dimensional support.  Then $r(F) = r(E)$ and $c_1(F) < c_1(E)$ since $T$ has $1$-dimensional support.  We also have $h^0(F)=h^0(E)$.  We let $\mu'$, $\alpha'$ and $\beta'$ be as in the previous proof.
If $\alpha'=\alpha$, then we can compute
$$
\beta-\beta' = r\alpha\left(\mu-\frac{1}{2}\alpha+\frac{3}{2}\right)-r\alpha\left(\mu'-\frac{1}{2}\alpha+\frac{3}{2}\right)= r\alpha(\mu-\mu')\geq \alpha.
$$If instead $\alpha'<\alpha$, then we can deduce the stronger inequality $\beta-\beta' > \alpha$  by arguing as in the final paragraph of the proof of Lemma \ref{lem-ggg}.

(1) Now suppose $E$ has deficiency less than $\alpha$, so $h^0(E) > \beta-\alpha.$  Then $h^0(F)\leq\beta' \leq \beta-\alpha <h^0(E)$, contradicting $h^0(F)=h^0(E)$.  Therefore $E$ is globally generated in codimension $1$.
 
 (2) If instead $h^0(E) = \beta-\alpha$, then we conclude that $h^0(F) \leq \beta' \leq \beta - \alpha =h^0(E)$, and all these inequalities are equalities.  But then $\beta-\beta'=\alpha$ forces $\alpha=\alpha'$ and $r(\mu-\mu')=1$, so $c_1(T) = 1$.  
\end{proof}

\begin{remark}
The lemma is sharp.  For example, consider collinear points $p_1,\ldots,p_n$ and the sheaf $E = I_{p_1,\ldots,p_n}(n-1)$.  Then $E$ has deficiency $\alpha = n$ and it is not globally generated in codimension $1$. 
\end{remark}

\subsection{Nonspeciality}

We next show that vector bundles with small deficiency must have vanishing first cohomology.  The proof is motivated by the results in \cite{LePotierAmple}.   

\begin{proposition}\label{prop-h1}
Let $E$ be a $\mu$-semistable vector bundle on $\P^2$ of rank $r$ and slope $\mu\geq 0$.  Suppose that $E$ is globally generated in codimension $1$ and that $h^0(E) \geq \beta - \alpha$.  Then $H^1(E) = 0$. 
 \end{proposition}
 Recall that Lemma \ref{lem-ggcodim1} (1) shows that the hypothesis that $E$ is globally generated in codimension $1$ is automatic if $h^0(E)>\beta-\alpha.$ 
\begin{proof}
The proof is by induction on $\alpha$, and we start with the  base case $\alpha = 1$.  Let $\delta$ be the deficiency of $E$, so that $\delta$ is either $0$ or $1$.  Let $L$ be any line and consider the restriction sequence $$0\to E(-1)\to E\to E|_L\to 0.$$ By semistability, we have $h^0(E(-1)) = 0$.  Also, since $E$ is globally generated in codimension $1$, we have $h^0(E|_L) = \beta$ and $h^1(E|_L) = 0$.  Then from the exact sequence $$0\to H^0(E) \to H^0(E|_L)\to H^1(E(-1))\to H^1(E) \to 0$$ we see that $h^1(E(-1))-h^1(E) = \delta$.  Since $H^1(E|_L) = 0$ for \emph{every} line $L$, the maps $H^1(E(-1))\to H^1(E)$ fit together to give a surjective map of vector bundles $$\OO_{\P^{2*}}(-1)\te H^1(E(-1))\to \OO_{\P^{2*}} \te H^1(E)\to 0$$ on the dual projective plane.   Letting $a = h^1(E)$ so $a+\delta = h^1(E(-1))$, this is a surjective map $$\OO_{\P^{2*}}(-1)^{a+\delta} \to \OO_{\P^{2*}}^a\to 0.$$ If $\delta=0$, the map is an isomorphism, which forces $a=0$.  If $\delta = 1$, the kernel would be the line bundle $\OO_{\P^{2*}}(-a-\delta)$ and we would have the exact sequence $$0\to \OO_{\P^{2*}}(-a-\delta) \to \OO_{\P^{2*}}(-1)^{a+\delta} \to \OO_{\P^{2*}}^a\to 0.$$ Computing Euler characteristics gives a contradiction unless $a=0$, since a line bundle on $\P^{2*}$ always has nonnegative Euler characteristic.   We conclude $h^1(E) = 0$ in either case.

Now suppose $\alpha > 1$ and $\delta \leq \alpha$.  We again consider the restriction sequence $$0\to E(-1)\to E \to E|_L\to 0.$$  We have $h^1(E|_L)=0$ since $E$ is globally generated in codimension $1$, so we get the exact sequence $$0\to H^0(E(-1))\to H^0(E)\to H^0(E|_L)\to H^1(E(-1))\to H^1(E)\to 0.$$  Let $\kappa = h^1(E(-1))-h^1(E)$.  Since $h^0(E)= \beta-\delta$ and $h^0(E|_L) = r(\mu+1)$, we compute $$h^0(E(-1)) = h^0(E)-h^0(E|_L)+\kappa = \beta-\delta-r(\mu+1)+\kappa = \beta_{r,\mu-1}-\delta+\kappa.$$
That is, $E(-1)$ has deficiency $\delta-\kappa$.  There are two cases to consider based on whether $\kappa \geq 2$ or $0\leq \kappa\leq 1$.

\emph{Case 1: $\kappa \geq 2$}.  This case actually never arises.  In this case, $E(-1)$ has deficiency less than $\alpha-1$.  By Lemma \ref{lem-ggcodim1}, $E(-1)$ is globally generated in codimension $1$.  By induction on $\alpha$, we find $h^1(E(-1))=0$, contradicting $\kappa \geq 2$.

\emph{Case 2: $0\leq \kappa \leq 1$.} This case is similar to the base case of the induction.  Let $a= h^1(E)$.  Then the maps $H^1(E(-1))\to H^1(E)$ fit together to give a surjection $$\OO_{\P^{2*}}(-1)^{a+\kappa}\to \OO_{\P^{2*}}^a\to 0.$$ As above, this forces $a=0$ and $h^1(E)=0$.
\end{proof}

The proposition allows us to prove a stronger statement about global generation. 

\begin{corollary}\label{cor-ggbundle}
Let $E$ be a $\mu$-semistable vector bundle on $\P^2$ of rank $r$ and slope $\mu\geq 0$.  Suppose that $E$ is globally generated in codimension $1$ and that $h^0(E) \geq \beta - \alpha$.  Then $E$ is actually globally generated.
\end{corollary}
\begin{proof}
We saw in the proof of Proposition \ref{prop-h1} that $\epsilon = h^1(E(-1))$ must be either $0$ or $1$.  We compute $E$ by using the Beilinson spectral sequence for $E$ corresponding to the exceptional collection  $\OO_{\P^2}(-2),\OO_{\P^2}(-1),\OO_{\P^2}$ with dual collection $\OO_{\P^2}(-1),T_{\P^2}(-2),\OO_{\P^2}$.  The $E_1^{p,q}$-page of this sequence looks like $$\xymatrix{
\OO(-2) \te H^1(E\te \OO(-1)) \ar[r] & \OO(-1) \te H^1(E\te T(-2))  \ar[r] & \OO \te H^1(E\te \OO)  \\
\OO(-2) \te H^0(E\te \OO(-1)) \ar[r] & \OO(-1) \te H^0(E\te T(-2))  \ar[r] & \OO \te H^0(E\te \OO) \\
}$$ and the sequence converges to $E$ in degree $0$.  Our known information shows that this page takes the shape $$\xymatrix{
\OO(-2)^\epsilon \ar[r] & \OO(-1) ^n  \ar[r] & 0 \\
\OO(-2)^l \ar[r] & \OO(-1)^m  \ar[r] & \OO^{\beta-\delta}  \\
}$$ for some integers $l,m,n$.  This shows that the cokernel $Q$ of the canonical evaluation $H^0(E)\te \OO_{\P^2}\to E$ is isomorphic to the cokernel of $\OO(-2)^\epsilon\to \OO(-1)^n$.  But $Q$ is at most zero-dimensional since $E$ is globally generated in codimension $1$.  Since $\epsilon$ is $0$ or $1$, this is only possible if $n=0$ and $Q=0$.
\end{proof}

Finally, we can prove analogous results for sheaves by resolving their singularities.

\begin{proposition}\label{prop-h1sheaf}
Let $E$ be a $\mu$-semistable \emph{sheaf} on $\P^2$ of rank $r$ and slope $\mu\geq 0$.  Suppose that $E$ is globally generated in codimension $1$ and that $h^0(E) \geq \beta-\alpha$.  Then $H^1(E)=0$ and $E$ is globally generated.
\end{proposition}
\begin{proof}
Since the result is known for vector bundles by Proposition \ref{prop-h1} and Corollary \ref{cor-ggbundle}, we assume that $E$ is not locally free and let $E^{**}$ be its double dual.  Then the cokernel $T$ of the inclusion $$0\to E\to E^{**}\to T\to 0$$ is zero-dimensional.  We induct on the length of $T$.

Since $T\neq 0$, it fits into an exact sequence $$0\to \OO_p\to T\xrightarrow{\psi} T'\to 0.$$  Let $E'$ be the kernel of the composition $E^{**}\to T\xrightarrow{\psi} T'$, which is a ``partial resolution'' of the singularities of $E$.  It is again $\mu$-semistable, and it has the same double dual $(E')^{**} = E^{**}$.  Using the snake lemma on the commutative diagram $$
\xymatrix{0\ar[r]&E\ar[r]\ar[d]_{\varphi} & E^{**} \ar[r]\ar@{=}[d]& T \ar[d]_{\psi} \ar[r]& 0\\
        0\ar[r]&E'\ar[r] & E^{**} \ar[r] & T'\ar[r]& 0
}
$$ with exact rows, we get an exact sequence $$0\to \ker\psi\to \coker\varphi\to 0.$$ This shows that $\coker\varphi\cong\OO_p$ and thus there is an exact sequence $$0\to E\xrightarrow{\varphi} E'\to \OO_p\to 0.$$  We have $h^0(E') \geq h^0(E)$, so the deficiency of $E'$ is at most the deficiency of $E$.  It also shows that since $E$ is globally generated in codimension $1$, so is $E'$. The exact sequence $$0\to E'\to E^{**}\to T'\to 0$$ then shows that our induction hypothesis applies to $E'$.  Therefore $h^1(E')=0$ and $E'$ is globally generated.  But then the map $H^0(E')\to H^0(\OO_p)$ is surjective, so $h^1(E)=0$.

It remains to prove $E$ is globally generated.  By the same argument as in the proof of Corollary \ref{cor-ggbundle}, it is enough to show that $h^1(E(-1))\leq 1$.  Taking cohomology of the exact sequence $$0\to E(-1)\to E'(-1)\to \OO_p\to 0,$$ we see that it is enough to show that $h^1(E'(-1))=0$.  For any line $L$, we have \begin{align*}
h^0(E'(-1))& \geq \chi(E'(-1)) \\&= 1+\chi(E(-1)) \\&= 1+\chi(E)-\chi(E|_L) \\&= 1+h^0(E)-r(\mu+1)
\\&\geq 1+ \beta_{r,\mu}-\alpha_\mu-r(\mu+1)\\
&= \beta_{r,\mu-1} -\alpha_{\mu-1}.
\end{align*}
Therefore, $E'(-1)$ has deficiency at most $\alpha_{\mu-1}$, and if equality holds, then $h^1(E'(-1))=0$ and we are done.  On the other hand, if the deficiency of $E'(-1)$ is less than $\alpha_{\mu-1}$, then $E'(-1)$ is globally generated in codimension $1$ by Lemma \ref{lem-ggcodim1} (1).  In this case our induction hypothesis shows $h^1(E'(-1))=0$.
\end{proof}

The proof of Theorem \ref{thm-h1} now follows from Proposition \ref{prop-h1sheaf} and Lemma \ref{lem-ggcodim1} (1).

\section{Steiner bundles}  \label{sec-steiner}

Our $h^1$ vanishing result allows us to classify the semistable maximal sheaves, i.e. the sheaves with $h^0(E)  = \beta$.  These turn out to always be twisted Steiner bundles.  For our work in the next section we will need to classify sheaves with $\beta$ sections under a slightly weaker hypothesis than stability as well.

\begin{definition} Any sheaf $E$ described as a cokernel of a map $$0\to \OO(-1)^{a}\fto{M} \OO^{a+r} \to E\to 0$$ is called a \emph{Steiner sheaf}.  The  sheaf $E$ has rank $r$ and slope $\mu = \frac{a}{r}$.  A \emph{twisted Steiner sheaf} is a sheaf of the form $$0\to \OO(k-1)^a\fto{M} \OO(k)^{a+r} \to E\to 0$$ with rank $r$ and slope $\mu = k+ \frac{a}{r}$.
\end{definition}  

\begin{remark}\label{rem-stabSteiner} Let us discuss the stability of a Steiner sheaf $E$ defined by a \emph{general} matrix of linear forms $M$, following results from \cite{Drezet} and \cite{Brambilla2}.  If $r\geq 2$, then $E$ defined by $$0\to \OO(-1)^a \fto{M} \OO^{a+r}\to E\to 0$$ is in fact a bundle.  If $$\mu =\frac{a}{r} > \varphi-1 \approx 0.618 \qquad\qquad \varphi = \frac{1+\sqrt{5}}2,$$ then $E$ is $\mu$-stable, and $E$ is a general bundle in a positive-dimensional moduli space.

If instead $\mu<\varphi-1$ and $E$ is stable, then it is a \emph{Fibonacci bundle.}   If we declare $F_{-1} = \OO(-1)$ and $F_0 = \OO$, then subsequent Fibonacci bundles are obtained via a mutation $$0\to F_{i-1} \to \Hom(F_{i-1},F_{i})^*\te F_{i}\to F_{i+1} \to 0.$$ For example, for $i=0$ we get the Euler sequence defining $F_1 = T_{\P^2}(-1)$.  Let $f_0 = 0$, $f_1 =1$, and $f_{k+2} = f_{k+1}+f_k$ be the Fibonacci sequence.  The Fibonacci bundles $F_i$ have invariants $$r(F_i) = f_{2i+1}\qquad c_1(F_i) = f_{2i} \qquad \mu(F_i) = \frac{f_{2i}}{f_{2i+1}} \qquad \chi(F_i) = \beta_{r(F_i),\mu(F_i)} \qquad \hom(F_{i-1},F_i) = 3.$$ The first several slopes of Fibonacci bundles are $$\frac{0}{1},\frac{1}{2},\frac{3}{5},\frac{8}{13},\frac{21}{34},\ldots$$ and these slopes converge to $\varphi-1$.   

Finally, if $M$ is general but $0\leq \mu <\varphi-1$, then there is a Fibonacci bundle $F_i$ such that $\mu(F_{i-1}) < \mu(E) \leq \mu(F_i)$.  Then there are integers $m\geq 0$ and $n>0$ (depending only on $a,r$) such that $E \cong F_{i-1}^{\oplus m} \oplus F_i^{\oplus n}$.
\end{remark}

\begin{proposition}\label{prop-Steiner}
Let $\alpha \geq 1$ and $0\leq a <r$ be integers.  Let $E$  be a twisted Steiner sheaf $$0\to \OO(\alpha-2)^{a}\fto{M} \OO(\alpha-1)^{a+r}\to E \to 0.$$ Then $E$ has slope $\mu = \alpha-1+\frac{a}{r}$.  We have $\alpha_\mu = \alpha$ and 
$h^0(E) = \chi(E)=\beta_{r,\mu}$ and $h^1(E) = 0$.  

If $r\geq2$ and $M$ is general, then $E$ is a bundle and $\alpha-1\leq \mumin(E)\leq \mumax(E)<\alpha$.
\end{proposition}
\begin{proof}
The computation of $\mu$ and $\alpha_\mu$ is straightforward.  It is then easy to verify that $h^0(E) = \beta$ and $E$ has no higher cohomology.  The second statement when $r\geq2$ and $M$ is general follows from the discussion before the proposition.  Either $\mu > \alpha+\varphi-2$ and $E$ is $\mu$-stable, or $\mu < \alpha + \varphi-2$ and $\mumax(E)$ is less than $\alpha+\varphi-2$.
\end{proof}

The bundles in the previous proposition can now be seen to be the basic building blocks for all sheaves $E$ which are not too unstable and have $h^0(E) = \beta$.

\begin{theorem}\label{thm-maxHN}
Let $E$ be a torsion-free sheaf with rank $r$ and slope $\mu\geq 0$ which has $$\alpha-2<\mumin(E)\leq \mumax(E)<\alpha$$ and $h^0(E) = \beta$.  Then $E$ is a bundle, and every Harder-Narasimhan factor $\gr_i$ of $E$ is a semistable twisted Steiner bundle of the form $$0\to \OO(\alpha-2)^{a_i}\to \OO(\alpha-1)^{a_i+r_i}\to \gr_i\to 0,$$ where $0\leq a_i < r_i$.

Conversely, if $\alpha\geq 1$ and $E$ is any bundle such that its Harder-Narasimhan factors $\gr_i$ all have the above form, then $\alpha-1\leq \mumin(E)\leq \mumax(E)<\alpha$ and $h^0(E) = \beta$.
\end{theorem}
\begin{proof}
First suppose that $E$ is a semistable sheaf with $h^0(E) = \beta$.  By Theorem \ref{thm-h1}, we find that $E$ has $h^1(E) = 0$ and $\chi(E) = \beta$.  If we write $\mu = \alpha-1+ \frac{a}{r}$ with $0\leq a <r$, then $E$ has the same Chern classes as a general twisted Steiner bundle of the form $$0\to \OO(\alpha-2)^a\to \OO(\alpha-1)^{a+r}\to F\to 0.$$  Dr\'ezet \cite{Drezet} shows that any semistable sheaf with these numerical invariants is in fact a twisted Steiner bundle of this form.  Furthermore, these bundles have the minimal possible discriminant among semistable sheaves of their rank and slope.  It follows by considering (Gieseker) Harder-Narasimhan filtrations that if instead we only assumed $E$ was $\mu$-semistable, then it would actually be semistable.

% First suppose that $E$ is a semistable \emph{bundle} with $h^0(E) = \beta$.  By Theorem \ref{thm-h1}, we find that $E$ has $h^1(E) = 0$ and $\chi(E) = \beta$.  If we write $\mu = k+ \frac{a}{r}$ with $k$ being an integer and $0\leq a <r$, then $E$ has the same Chern classes as a general twisted Steiner bundle of the form $$0\to \OO(k-1)^a\to \OO(k)^{a+r}\to F\to 0.$$ But Dr\'ezet \cite{Drezet} shows that any semistable sheaf with these numerical invariants is in fact a twisted Steiner bundle of this form.  Furthermore, these bundles have the minimal possible discriminant among semistable sheaves of their rank and slope.  It follows by considering (Gieseker) Harder-Narasimhan filtrations that if instead we only assumed $E$ was $\mu$-semistable, then it would actually be semistable.

% Next suppose that $E$ is a semistable \emph{sheaf} with $h^0(E) = \beta$  and that it is singular at at least one point.  Then it includes into its double dual with a 0-dimensional torsion cokernel $T$: $$0\to E \to E^{**}\to T\to 0.$$  The bundle $E^{**}$ is automatically $\mu$-semistable.  It has the same rank and $c_1$ as $E$, and at least as many sections, so in fact $h^0(E^{**}) = \beta$ by Corollary \ref{cor-sectionBound}.   By the previous paragraph, $E^{**}$ is a semistable twisted Steiner bundle.  But any such bundle is globally generated, while the sections of $E$ cannot generate all of $E^{**}$.  Thus $h^0(E) < h^0(E^{**})$, contradicting that both have $\beta$ sections.  Therefore $E$ is actually a bundle.

Now if $E$ is a torsion-free sheaf with $\alpha-2<\mumin(E)\leq \mumax(E)<\alpha$ and $h^0(E) = \beta$, then we examine the proof that $h^0(E) \leq \beta$ given in the proof of Corollary \ref{cor-sectionBound}.  In order for equality to hold in that proof, each $\alpha^i$ must equal $\alpha$ and each Harder-Narasimhan factor $\gr_i$ must have $h^0(\gr_i) = \beta^i$.  By our earlier analysis, each $\gr_i$ is a semistable twisted Steiner bundle of the required form.

For the converse, we note that since $h^1(\gr_i)=0$ for a twisted Steiner bundle $\gr_i$ we have $h^0(E) = \sum_i h^0(\gr_i)$.  Since each $\gr_i$ has $\beta_{r_i,\mu_i}$ sections, it follows that $E$ has $\beta$ sections.
\end{proof}

In fact, any extension of two twisted Steiner bundles with the same twist is again a twisted Steiner bundle with that twist.  

\begin{lemma}\label{lem-extension}
Let $E$ be any bundle fitting as an extension $$0\to S_1\to E\to S_2\to 0$$ where each $S_i$ is an (untwisted) Steiner bundle $$0\to \OO(-1)^{a_i}\to \OO^{a_i+r_i}\to S_i\to 0.$$ Then $E$ is also a Steiner bundle.
\end{lemma}
\begin{proof}
We use the Beilinson spectral sequence for the exceptional collection $\OO(-2),\OO(-1),\OO$ with dual collection $\OO(-1),T_{\P^2}(-2),\OO$ to compute $E$ (see the proof of Corollary \ref{cor-ggbundle}).  We compute $h^0(S_i) = a_i+r_i$ and $h^1(S_i)=0$, so $h^0(E) = a_1+a_2+r_1+r_2$ and $h^1(E) = 0$.  Since $T(-2)$ has no cohomology,  the cohomology of $S_i\te T(-2)$ is given by $$H^q(S_i\te T(-2)) \cong H^{q+1}(T(-3)^{a_i}),$$ so $h^0(S_i\te T(-2))=a_i$ and $h^1(S_i\te T(-2))=0$.  Thus $$h^0(E\te T(-2)) = a_1+a_2 \quad \mbox{and} \quad h^1(E\te T(-2))=0.$$  Clearly $S_i(-1)$ has no cohomology, so $E(-1)$ has no cohomology.  A resolution
$$0\to \OO(-1)^{a_1+a_2}\to \OO^{a_1+a_2+r_1+r_2}\to E\to 0$$ then follows from the Beilinson spectral sequence. 
\end{proof}

Combining Theorem \ref{thm-maxHN} and Lemma \ref{lem-extension} gives the following simpler statement.

\begin{corollary}\label{cor-maxSteiner}
Let $E$ be a torsion-free sheaf with rank $r$ and slope $\mu\geq 0$ which has $$\alpha-2<\mumin(E)\leq \mumax(E)<\alpha$$ and $h^0(E) = \beta$.  Then $E$ is a twisted Steiner bundle of the form $$0\to \OO(\alpha-2)^{a}\to \OO(\alpha-1)^{a+r}\to E\to 0,$$ where $0\leq a < r$.
\end{corollary}

\section{Deficiency $\alpha$}\label{sec-alpha}

\subsection{Classification} We have seen in Theorem \ref{thm-h1} that if $E$ is a semistable bundle with deficiency $\alpha$ which is globally generated in codimension $1$, then $h^1(E) = 0$.  We consider these to be the ``primitive'' examples of bundles of deficiency $\alpha$, and they are somewhat rare.

On the other hand, if we omit the hypothesis that $E$ is globally generated in codimension $1$, then examples with deficiency $\alpha$ are ubiquitous.  However, they all arise from twisted Steiner bundles via elementary transformations along a line.

\begin{theorem}\label{thm-defAlpha}
Let $E$ be a $\mu$-semistable sheaf on $\P^2$ of rank $r$ and slope $\mu> 0$ having $h^0(E) = \beta-\alpha$.  Suppose that $E$ fails to be globally generated in codimension $1$, and write $\mu = \alpha-1+\frac{a}{r}$ where $0\leq a <r$.  \begin{enumerate}
\item If $a>0$, then there is a twisted Steiner bundle $S$ of the form $$0\to \OO(\alpha-2)^{a-1}\to \OO(\alpha-1)^{a+r-1}\to S\to 0$$ such that $E$ fits as an extension $$0\to S\to E\to \OO_L(-b)\to 0$$ for some line $L$, where $b=1+h^1(E)$.
\item If $a=0$ and $E$ is Gieseker semistable, then we must have $r=1$ and $E = I_Z(\alpha-1)$ for a collinear $0$-dimensional subscheme $Z\subset \P^2$ of length at least $\alpha$. \end{enumerate}
\end{theorem}
\begin{proof}
Let $S$ be the image of the canonical evaluation $H^0(E)\te E\to E$.  Then we are in the situation of the proof of Lemma \ref{lem-ggcodim1} (2), and there is an exact sequence $$0\to S\to E\to T\to 0$$ where $T$ is torsion and has $c_1(T) = 1$.  Then $S$ has rank $r$ and slope $\mu'=\mu - \frac{1}{r}$.

Since $E$ is $\mu$-semistable, $S$ is not too far from being stable.  More precisely, we have $$\mumax(S) \leq \mu = \mu' + \frac{1}{r},$$ and $\mumin(S)$ will be as small as possible if the maximal destabilizing subsheaf has rank $r-1$ and slope $\mu$.  This gives $\mumin(S) \geq \mu-1$.  

(1) Now suppose that $a>0$.   This gives $\alpha_{\mu'}=\alpha$ and $\mumin(S) > \alpha-2$.  We compute $$\beta_{r,\mu'} = r\alpha\left(\mu'-\frac{1}{2}\alpha+\frac{3}{2}\right)=r\alpha\left(\mu-\frac{1}{r}-\frac{1}{2}\alpha+\frac{3}{2}\right)=\beta_{r,\mu}-\alpha,$$ so that $S$ has deficiency $0$ since $h^0(S) = h^0(E) = \beta_{r,\mu}-\alpha$.  Corollary \ref{cor-maxSteiner} shows that $S$ is a twisted Steiner bundle of the required form.  We now analyze $T$.  Since $H^1(S)=0$ and $H^0(S)=H^0(E)$, we find that $T$ has no global sections.  Therefore it does not have any $0$-dimensional subsheaf, and $T$ is pure of dimension $1$.  Since $c_1(T) = 1$, it is a line bundle supported on a line $L$, and since it has no global sections we find $T \cong \OO_L(-b)$ for some $b>0$.

(2) Suppose instead that $a=0$ and $E$ is Gieseker semistable.   This time we have $\alpha_{\mu'} = \alpha-1$.  Let $\gr_1,\ldots,\gr_\ell$ be the Harder-Narasimhan factors of $S$.  All the factors have slopes $\mu_i$ satisfying $\alpha-2\leq \mu_i \leq \alpha-1$.  For each factor with $\mu_i < \alpha-1$, we can estimate $$\beta_{r_i,\mu_i} = r_i(\alpha-1)\left(\mu_i-\frac{1}{2}(\alpha-1)+\frac{3}{2}\right) \leq r_i\alpha\left(\mu_i -\frac{1}{2}\alpha+\frac{3}{2}\right),$$ and the inequality is strict unless $\mu_i=\alpha-2$.  Using Theorem \ref{thm-sectionBound} we then get $$ h^0(S) \leq \sum_i h^0(\gr_i) \leq \sum_i \beta_{r_i,\mu_i} \leq \sum_i r_i\alpha\left(\mu_i - \frac{1}{2} \alpha + \frac{3}{2}\right) = r\alpha\left(\mu'-\frac{1}{2}\alpha+\frac{3}{2}\right)=\beta_{r,\mu}-\alpha,$$ and since $h^0(S) = \beta_{r,\mu}-\alpha$ we learn that every $\mu_i$ must equal either $\alpha-2$ or $\alpha-1$ and every factor $\gr_i$ is a maximal sheaf.  For $S$ to have slope $\alpha-1-\frac{1}{r}$, the only possibility is that $S$ has Harder-Narasimhan factors $\OO(\alpha-1)^{\oplus(r-1)}$ and $\OO(\alpha-2)$.  But then $r=1$ as otherwise $\OO(\alpha-1)\to E$ Gieseker destabilizes $E$. Thus, $E=I_Z(\alpha-1)$ for some $0$-dimensional subscheme $Z\subset \P^2$ and we have an exact sequence $$0\to \OO(\alpha-2)\to I_Z(\alpha-1)\to \OO_L(-b)\to 0$$ for some $b>0$, where $L$ is the line that $T$ is supported on. Twisting by $\OO(-\alpha+1)$, we get the exact sequence $$0\to \OO(-1)\to I_Z\to \OO_L(-b-\alpha+1)\to 0.$$ It follows that a linear form defining $L$ must vanish on $Z$, which shows that $Z$ must be collinear. Since $h^0(E)=h^0(\OO(\alpha-1))-\alpha$, the length of $Z$ must be at least $\alpha$.
\end{proof}

This classification result allows us to explicitly describe the Brill-Noether loci $B^{\beta-\alpha}$ whenever they are nonempty.  We will take up the question of nonemptiness in the next subsection.

\begin{corollary}\label{cor-irr}
If the Brill-Noether locus $B^{\beta-\alpha} \subset M(r,\mu,\chi)$ is nonempty, then it is irreducible.
\end{corollary}
\begin{proof}
   If $\chi \geq \beta-\alpha$, then $B^{\beta-\alpha} = M(r,\mu,\chi)$ and there is nothing to prove.  So instead we focus on the case $\chi < \beta - \alpha$.  Then every sheaf with at least $\beta-\alpha$ sections has nonzero $h^1$, so must fail to be globally generated in codimension $1$ by Theorem \ref{thm-h1}.  Thus Theorem \ref{thm-defAlpha} describes $E$.  We may as well assume $r\geq 2$ as the $r=1$ case is easy.  If we let $b = \beta-\alpha-\chi+1$, then every $E\in B^{\beta-\alpha}$ must fit in an exact sequence $$0\to S\to E\to \OO_L(-b)\to 0.$$

Keeping the notation from Theorem \ref{thm-defAlpha} (1), let $$U\subset \Hom(\OO(\alpha-2)^{a-1},\OO(\alpha-1)^{a+r-1}) \times \P^{2*}$$ be the open subset parameterizing pairs $(\phi,L)$ where $\phi$ is a homomorphism with locally free cokernel $S_\phi$.  The dimension \begin{align*}\ext^1(\OO_L(-b),S_\phi) &= \ext^1(S_\phi,\OO_{L}(-b-3)) \\&= h^1(S_\phi^*(-b-3)|_L)\\&=h^0(S_\phi(b+1)|_L) \\&= r(\mu+b+2)-1\end{align*} is independent of the choice of $(\phi,L)\in U$, so over $U$ we can construct a projective space bundle $X\to U$ whose fiber over $(\phi,L)\in U$ is $\P\Ext^1(\OO_L(-b),S_\phi)$.  This space carries a universal extension bundle, and if $B^{\beta-\alpha}$ is nonempty, then the general sheaf parameterized in this way is in $B^{\beta-\alpha}$.  Thus the induced rational mapping $X\dashrightarrow M(r,\mu,\chi)$ dominates the locus $B^{\beta-\alpha}$, and $B^{\beta-\alpha}$ is irreducible.
\end{proof}

\subsection{Existence}

We now investigate when the Brill-Noether locus $B^{\beta-\alpha} \subset M(r,\mu,\chi)$ is nonempty.  The following theorem summarizes the results in this subsection.  We remind the reader that $\varphi$ is the golden ratio $(1+\sqrt{5})/2$.

\begin{theorem}\label{thm-existence}
Let $\mu> 0$ be such that $\mu = \alpha-1+\frac{a}{r}$ with $0< a<r$, and assume $\chi<\beta-\alpha$.
\begin{enumerate}
\item If $\frac{a}{r} > \varphi-1$, then $B^{\beta-\alpha} \subset M(r,\mu,\chi)$ is nonempty.
\item If $\frac{a}{r} < \varphi-1$, let $F_i$ be the Fibonacci bundle such that   $\mu(F_{i-1}) < \frac{a-1}{r} \leq \mu(F_i)$.
\begin{enumerate}
\item If $\frac{a}{r} \leq \mu(F_i)$, then $B^{\beta-\alpha}\subset M(r,\mu,\chi)$ is empty.
\item If $\frac{a}{r} > \mu(F_i)$, then $B^{\beta-\alpha} \subset M(r,\mu,\chi)$ is nonempty if $\chi \ll 0$.
\end{enumerate}
\end{enumerate}
\end{theorem}

Throughout this subsection we let $\mu > 0$ be a non-integer and write it as $\mu = \alpha-1+\frac{a}{r}$, and assume $\chi < \beta-\alpha$.  When there is a $\mu$-stable twisted Steiner bundle of rank $r$ and slope $\mu$, the existence problem is easy and taken care of by the following proposition. This proves case (1) of Theorem \ref{thm-existence}.

\begin{proposition}
Assume that $\frac{a}{r} > \varphi-1.$  Then  $B^{\beta-\alpha}\subset M(r,\mu,\chi)$ is nonempty. 
\end{proposition}
\begin{proof}
We use a $\mu$-stable twisted Steiner bundle of slope $\mu$ and rank $r$ to construct a $\mu$-stable sheaf $E$ with $\beta-\alpha$ sections directly.
Let $F$ be a general twisted Steiner bundle of the form $$0\to \OO(\alpha-2)^a\to \OO(\alpha-1)^{a+r} \to F\to 0.$$ Then $h^0(F) = \beta = \chi(F)$ and the assumption $\frac{a}{r} > \varphi-1$ implies that $F$ is $\mu$-stable.  We now modify $F$ to decrease the Euler characteristic to $\chi$ and the number of sections to $\beta-\alpha$.

Since $F$ has slope between $\alpha-1$ and $\alpha$ and $F(-\alpha+1)$ is globally generated, the restriction $F|_L$ splits as a direct sum of line bundles on $L$ where at least one of the line bundles is $\OO_L(\alpha-1)$.  Thus there is a quotient $$F|_L \to \OO_L(\alpha-1).$$  Let $Z$ be a collection of $\beta-\chi$ points on $L$, and let $$\OO_L(\alpha-1)\to \OO_Z$$  be the restriction.  Let $\phi:F\to \OO_Z$ be the composition $$F\to F|_L\to \OO_L(\alpha-1)\to \OO_Z.$$ Then we consider the kernel $F'$ given by $$0\to F'\to F\fto\phi \OO_Z\to 0.$$ We observe $\chi(F') = \chi$ and $\mu(F') = \mu$.  Since $F$ is $\mu$-stable, it is easy to see that $F'$ is also $\mu$-stable, as a destabilizing subsheaf of $F'$ would also destabilize $F$.

Finally we verify that $h^0(F') = \beta - \alpha$.  Since $h^0(F) = \beta$, this amounts to showing that the map on global sections $H^0(F) \to H^0(\OO_Z)$ has rank $\alpha$.  But this map is the composition of maps $$H^0(F) \onto H^0(F|_L) \onto H^0(\OO_L(\alpha-1)) \hookrightarrow H^0(\OO_Z),$$ of which the first two maps are surjective (since $H^1(F(-1)) = 0$) and the last is injective, as indicated.  Since $h^0(\OO_L(\alpha-1)) = \alpha$, the rank is $\alpha$.
\end{proof}

The case where $0 < \frac{a}{r} < \varphi-1$ is more challenging.  If we write $b=\beta-\alpha-\chi+1$, then Theorem \ref{thm-defAlpha} (1) shows that any $E\in B^{\beta-\alpha}$ must fit as an extension of the form 
$$0\to S\to E\to \OO_L(-b)\to 0,$$ where $S$ is a twisted Steiner bundle of the form $$0\to \OO(\alpha-2)^{a-1}\to \OO(\alpha-1)^{a+r-1}\to S\to 0.$$  Thus the existence problem is equivalent to determining when there are \emph{semistable} sheaves $E$ described by extensions of this form.  Since all these extensions can be parameterized by an irreducible variety, the openness of semistability shows that this is equivalent to the general extension $E$ being semistable.  Thus we seek to understand when $E$ is semistable if the bundle $S$, the line $L$, and the extension class $e\in \Ext^1(\OO_L(-b),S)$ are all general.

We can first show that $E$ is always torsion-free if the extension is nontrivial.

\begin{lemma}\label{lem-torsionFree}
Let $$0\to S\to E\to \OO_L(-b)\to 0$$ be a non-split extension.  Then $E$ is torsion-free.
\end{lemma}
\begin{proof}
Suppose the torsion subsheaf $T$ of $E$ is nonzero, and let $E' = E/T$.  Since $S$ is a bundle, the composition $T\to E\to \OO_L(-b)$ is injective with some image $\OO_L(-c) \subset \OO_L(-b)$.    If $c=b$ then the sequence splits, so suppose $c > b$.  Then the snake lemma gives the following diagram:
$$
\xymatrix{& & 0\ar[d] & 0\ar[d]\\
& & T \ar@{=}[r]\ar[d]& \OO_L(-c) \ar[d] & \\
0\ar[r]&S\ar[r]\ar@{=}[d] & E \ar[r]\ar[d] & \OO_L(-b)\ar[r]\ar[d]& 0\\
0\ar[r]&S\ar[r]&E'\ar[r]\ar[d] & \OO_Z \ar[r]\ar[d] &0\\
&&0&0
}
$$
Here $\OO_Z$ is zero-dimensional of length $c-b$.  Tensoring the sequence $$0\to S\to E'\to \OO_Z\to 0$$ by $\OO(-d)$ for large $d$, taking cohomology, and using Serre duality and Serre vanishing contradicts that $c-b>0$.
\end{proof}

Now we fix a non-split extension $E$.  For studying semistability, we may as well twist by $\OO(-\alpha+1).$ Thus we may assume $\alpha = 1,$ $\mu(E) = \frac{a}{r}$,  and $S$ is a general bundle of the form $$0\to \OO(-1)^{a-1}\to \OO^{a+r-1}\to S\to 0.$$  We have $0\leq \frac{a-1}{r}< \varphi-1$, so there is a Fibonacci bundle $F_i$ such that $\mu(F_{i-1}) < \mu(S) \leq \mu(F_{i})$ (see Remark \ref{rem-stabSteiner}).  Then there are integers $m\geq 0$ and $n> 0$ such that $S \cong F_{i-1}^{\oplus m} \oplus F_{i}^{\oplus n}$.  We then have the following negative result, which proves Theorem \ref{thm-existence} (2(a)).

\begin{proposition}
We use the notation of the previous paragraph.  If $\mu(E) \leq \mu(F_i)$, then $E$ is never semistable.
\end{proposition}
\begin{proof}
Since $S \cong F_{i-1}^{\oplus m} \oplus F_i^{\oplus n}$ and $n>0$, we see that $F_i$ is a subbundle of $E$.  If $\mu(E)<\mu(F_i)$ then $F_i$ slope-destabilizes $E$.  If instead $\mu(E) = \mu(F_i)$ and $E$ is semistable, then $E$ would have to be a direct sum of copies of $F_i$.  But then $E$ would have $\beta$ sections instead of $\beta-\alpha$ sections.
\end{proof}

In the other case $\mu(E) > \mu(F_i)$, we develop a criterion to test the semistability of an extension.

\begin{proposition}
With the previous notation, suppose that $\mu(E) > \mu(F_i)$.  Let $F_j$ be the Fibonacci bundle of largest slope such that $\mu(F_j) < \mu(E)$.  Then $E$ is semistable if and only if $\Hom(E,F_j)=0$.
\end{proposition}
\begin{proof} Clearly $E$ is not semistable if $\Hom(E,F_j)\neq 0$.  Suppose the extension $E$ $$0\to S\to E\to \OO_L(-b)\to 0$$ is not semistable, where $S \cong F_{i-1}^{\oplus m} \oplus F_i^{\oplus n}$ as above.  The assumption $\mu(E)>\mu(F_i)$ gives $\mu(E) > \mumax(S)=\mu(F_i)$. We define a destabilizing subsheaf $E'$ and the quotient $Q$ $$0 \to E'\to E \to Q\to 0$$ as follows.
 \begin{itemize}
 \item If $\mumin(E)=\mumax(E)$, then we let $E'\subset E$ be the maximal Gieseker destabilizing sheaf.  
\item If $\mumin(E)<\mumax(E)$, then we let $E'\subset E$ be the maximal $\mu$-destabilizing subsheaf. In this case we have $\mu(E') >  \mumax(Q)$.
 \end{itemize}

 The composition $E'\to E\to \OO_L(-b)$ must be nonzero, since $E'$ is not a subsheaf of $S$ by stability.  The image of this composition is thus a sheaf of the form $\OO_L(-c)$ for some $c\geq b$.  Let $S'$ be the kernel of $E'\to \OO_L(-c)$.  Then $S'$ is a subsheaf of $S$, and we let $S''$ be the cokernel of $S'\to S$.  We let $T$ be the cokernel of $\OO_L(-c)\to \OO_L(-b)$. The snake lemma then gives the exact sequence in the final row of the following diagram.
$$
\xymatrix{&0\ar[d] & 0\ar[d] & 0\ar[d]\\
0\ar[r]&S'\ar[r]\ar[d] & E' \ar[r]\ar[d]& \OO_L(-c) \ar[d]\ar[r] & 0\\
0\ar[r]&S\ar[r]\ar[d] & E \ar[r]\ar[d] & \OO_L(-b)\ar[r]\ar[d]& 0\\
0\ar[r]&S''\ar[r]\ar[d]&Q\ar[r]\ar[d] & T \ar[r]\ar[d] &0\\
&0&0&0
}
$$

\emph{Step 1: $E'$ has slope $\mu(E')\leq 1$.}  Recall that $E'$ is $\mu$-semistable.   We have $\mu(S') = \mu(E') - \frac{1}{r(E')}$.  Since $\mumax(S) = \mu(F_i)$, we find $\mumax(S') \leq \mu(F_i)$ and thus $\mu(S') \leq \mu(F_i)$.  Then $\mu(E') \leq \mu(F_i)+\frac{1}{r(E')}$.  If $r(E') \geq 3$ this immediately gives $\mu(E') < 1$ since $\mu(F_i) < \varphi-1$.  If instead $r(E')$ is $1$ or $2$, then additional integrality considerations show $\mu(E') \leq 1$.

\emph{Step 2: $S'$ has $0\leq \mu(S') \leq \mu(F_i)$ and $-1<\mumin(S')\leq \mumax(S') \leq \mu(F_i)$.}  As in the first part of the proof of Theorem \ref{thm-defAlpha}, we see that $\mumin(S') \geq \mu(E')-1>-1$ since $E'$ is semistable.  Since $S'$ is a subsheaf of $S$, it has $\mu(S') \leq \mumax(S') \leq \mumax(S) = \mu(F_i)$.  The inequality $0\leq \mu(S')$ follows from $\mu(S') = \mu(E') - \frac{1}{r(E')}$ and $\mu(E') \geq \mu(E) > 0$.

\emph{Step 3: $E'$ has $h^0(E') \leq r(E')+c_1(E')-1$.}  Since $c>0$, we have $h^0(E') = h^0(S')$.  From Step 2 and Corollary \ref{cor-sectionBound}, we get $$h^0(S') \leq \beta_{r(S'),\mu(S')} = r(S')+c_1(S') = r(E')+c_1(E') - 1.$$

\emph{Step 4: $Q$ has $\mumax(Q)<1$ and $h^0(Q) \leq r(Q)+c_1(Q)$.}  Since $S$ is globally generated, $E$ is generically globally generated, and $Q$ is generically globally generated.  Therefore $\mumin(Q) \geq 0$.  By construction we have $\mumax(Q) \leq \mu(E')$, and if $\mu(E')=1$, then the inequality is strict.  Therefore $\mumax(Q) <1$ due to Step 1.  Corollary \ref{cor-sectionBound} now implies $h^0(Q) \leq \beta_{r(Q),\mu(Q)}=r(Q)+c_1(Q).$

\emph{Step 5: In fact, $h^0(E') = r(E')+c_1(E')-1$ and $h^0(Q) = r(Q)+c_1(Q)$.} We have $$h^0(E) = h^0(S) = r(E) + c_1(E) -1,$$ but Steps 3 and 4 give 
$$h^0(E)\leq h^0(E')+h^0(Q) \leq r(E')+c_1(E')-1+r(Q)+c_1(Q)=r(E)+c_1(E)-1.$$ Therefore equality holds in Steps 3 and 4.

\emph{Step 6: $S''$ has $h^0(S'') = r(S'')+c_1(S'')$.}  Because $S'$ and $S$ have $h^0(S') = r(S')+c_1(S')$ and $h^0(S) = r(S) + c_1(S)$, we see that $$h^0(S'') \geq h^0(S)-h^0(S') = r(S'')+c_1(S'').$$ On the other hand, $r(S'') = r(Q)$ and $c_1(S'') = c_1(Q)$ so $$h^0(S'') \leq h^0(Q) = r(Q)+c_1(Q)=r(S'')+c_1(S'').$$
Thus equality holds.

\emph{Step 7: $S'$ and $S''$ are Steiner bundles, $T=0$, $Q \cong S''$, and $b=c$ .}  Since $S$ and $Q$ are torsion-free, so are $S'$ and $S''$. By Step 2, we have $\mu(S')\geq0$ and $-1<\mumin(S')\leq \mumax(S') <1$. Since $S$, and therefore $S''$, is globally generated, we get $\mumin(S'')\geq0$. Also, $\mumax(S'')\leq \mumax(Q)<1$ by Step 4. Corollary \ref{cor-maxSteiner} now shows that $S'$ and $S''$ are Steiner bundles.  If we twist the sequence $$0\to S''\to Q \to T\to 0$$ by $\OO(-d)$ with $d\gg 0$, then $H^0(Q(-d)) = H^1(S''(-d))=0$ since $S''$ is a vector bundle.  Therefore $H^0(T)=0$, which gives $T =0$ since $T$ is zero-dimensional.  At this point our diagram reduces to the following.
$$
\xymatrix{&0\ar[d] & 0\ar[d] \\
0\ar[r]&S'\ar[r]\ar[d] & E' \ar[r]\ar[d]& \OO_L(-b) \ar@{=}[d] \ar[r]& 0\\
0\ar[r]&S\ar[r]\ar[d] & E \ar[r]\ar[d] & \OO_L(-b)\ar[r]& 0\\
&Q\ar@{=}[r]\ar[d] & Q\ar[d]   \\
&0&0
}
$$

\emph{Step 8: We have $\Hom(E,F_j) \neq 0$.} Since $Q$ is a Steiner bundle and $\mumin(Q)<\varphi-1$, there is a quotient Fibonacci bundle $Q\onto F_k$ with $\mu(F_k) = \mumin(Q)$.  We therefore have $\Hom(E,F_k)\neq 0$.  If we have $\mu(F_k) = \mu(E)$, then $E\to F_k$ would not be a destabilizing quotient since $F_k$ has maximal $\chi/r$ among semistable sheaves its slope.  Therefore $\mu(F_k) < \mu(E)$ and $\mu(F_k) \leq \mu(F_j)$, so $k\leq j$.     The mutation sequences defining the Fibonacci bundles show that there are injections $F_{\ell}\hookrightarrow F_{\ell+1}^{\oplus 3}$ for each $\ell$.  Thus there is an injection $F_k\hookrightarrow F_{k+1}^{\oplus 3}$, and so $\Hom(E,F_{k+1})\neq 0$. Inducting, we get $\Hom(E,F_j) \neq 0$.
\end{proof}

The next lemma now completes the proof of Theorem \ref{thm-existence} (2(b)) when applied to $A= S$ and $B = F_j$.

\begin{lemma}
Suppose $A,B$ are bundles on $\P^2$ with $\Hom(A,B(-1))=0$.  If $b \gg 0$ and $$0\to A\to E \to \OO_L(-b)\to 0$$ is a general extension,  then $\Hom(E,B) = 0$.
\end{lemma}
\begin{proof}
We apply $\Hom(-,B)$ to the sequence to get $$0\to \Hom(E,B)\to \Hom(A,B)\fto\delta \Ext^1(\OO_L(-b),B).$$ We want to prove $\delta$ is injective.  The boundary map $\delta$ arises from the product $$m:\Ext^1(\OO_L(-b),A) \te \Hom(A,B)\to \Ext^1(\OO_L(-b),B)$$ by $$\delta(\phi) = m(e\te \phi),$$ where $e\in \Ext^1(\OO_L(-b),A)$ is the extension class defining $E$.

The space $\Ext^1(\OO_L(-b),A)$ can be computed as
\begin{align*}
\Ext^1(\OO_L(-b),A) &\cong \Ext^1(A,\OO_L(-b-3))^*\\
&\cong H^1(A^*(-b-3)|_L)^*\\
&\cong H^0(A(b+1)|_L),
\end{align*}
and similarly for $B$.

The assumption $\Hom(A,B(-1))=0$ shows that we can interpret $\Hom(A,B)$ as a subspace of $\Hom(A|_L,B|_L)\cong \Hom(A(b+1)|_L,B(b+1)|_L).$  Thus we can interpret $m$ as arising from the restriction of the natural map $$\overline m : H^0(A(b+1)|_L)\te \Hom(A(b+1)|_L,B(b+1)|_L)\to H^0(B(b+1)|_L).$$ If there is an $e\in H^0(A(b+1)|_L)$ such that contraction with $e$ gives an injective map \begin{align*}
\Hom(A(b+1)|_L,B(b+1)|_L)&\to H^0(B(b+1)|_L)
\end{align*}
then the corresponding $\delta$ will be injective and we are done.  Decomposing $A$ and $B$ into direct sums of line bundles, the vector $e$ can be viewed as a vector of homogeneous forms of some degrees that are increasing linearly with $b$.  A nonzero element $M\in \Hom(A(b+1)|_L,B(b+1)|_L)$ is a matrix of homogeneous forms of fixed (hence bounded) degrees.  If $b$ is sufficiently large and $e$ is sufficiently general, we may assume that every relation among the entries of $e$ has degree larger than the degrees of the entries of $M$.  Thus $Me \neq 0$ for every nonzero $M$.  
\end{proof}

\subsection{Dimension} We finish the paper by computing the dimension of $B^{\beta-\alpha} \subset M(r,\mu,\chi)$ in a simple case. 

\begin{proposition}\label{prop-dim}
Write $\mu = \alpha -1 + \frac{a}{r}$ with $0<a<r$, and suppose $\chi< \beta - \alpha$.  Assume that $\frac{a-1}{r} > \varphi-1 .$    Then $B^{\beta-\alpha}\subset M(r,\mu,\chi)$ has codimension $$\codim B^{\beta-\alpha} = r(\beta-\alpha-\chi)+c_1-1.$$ This is strictly smaller than the expected codimension $\ecodim B^{\beta-\alpha} =(\beta-\alpha)(\beta-\alpha-\chi)$ except in the case where $\alpha = 1$ and $\chi = \beta-\alpha-1$.
\end{proposition}
\begin{proof}  We keep the notation of the proof of Corollary \ref{cor-irr}, so any $E\in B^{\beta-\alpha}$ can be written in the form $$0\to S\to E \to \OO_L(-b)\to 0.$$  If $E$ is general, then $S$ is given by a general element of $U$, and thus $S$ is a stable twisted Steiner bundle because of the assumption $\frac{a-1}{r} > \varphi-1$.  The line $L$ is uniquely determined as the locus where $E$ fails to be globally generated.  We have $\Hom(S,\OO_L(-b)) = 0$, so $\hom (S,E)  = \hom(E,\OO_L(-b)) = 1$, and thus the exact sequence is essentially uniquely determined by $E$.  It follows that the class $e\in \Ext^1(\OO_L(-b),S)$ is uniquely determined up to scale.  Therefore \begin{align*}\dim B^{\beta-\alpha} &= \dim M(\ch(S)) +\dim \P^{2*}+ \dim \P \Ext^1(\OO_L(-b),S))\\
&= 1-\chi(S,S)+2-\chi(\OO_L(-b),S)-1\\
&= 2-\chi(E,S).\end{align*}
On the other hand $M(\ch(E))$ has dimension $1-\chi(E,E)$, so the codimension of $B^{\beta-\alpha}\subset M(\ch(E))$ is 
\begin{align*}\codim B^{\beta-\alpha} &= (1-\chi(E,E)) - (2-\chi(E,S)) \\&= \chi(E,S) - \chi(E,E)-1 \\&= -\chi(E,\OO_L(-b))-1\\
&= -\chi(E^*(-b)|_L)-1\\
&= \chi(E(b-2)|_L)-1\\
&= r(\mu+b-1)-1\\
&= r(\beta-\alpha-\chi)+c_1-1.\end{align*}

We always have $\codim B^{\beta-\alpha} \leq \ecodim B^{\beta-\alpha},$ or  $$r(\beta-\alpha-\chi)+c_1-1 \leq (\beta-\alpha)(\beta-\alpha-\chi),$$ because $B^{\beta-\alpha}$ is a determinantal variety.  Let us analyze when equality holds.  The inequality $\frac{a-1}{r} > \varphi-1$ forces $c_1\geq 2$, so the inequality $\codim B^{\beta-\alpha} \leq \ecodim B^{\beta-\alpha}$ gives  $r < \beta-\alpha$.  Since the inequality $\codim B^{\beta-\alpha} \leq \ecodim B^{\beta-\alpha}$ holds for every $\chi <\beta-\alpha$, equality can only possibly hold in case $\chi = \beta-\alpha-1$.  In this case the inequality becomes $$r+c_1-1 \leq \beta-\alpha.$$  When $\alpha=1$, both sides are equal, and it is easy to see by a direct computation that the inequality becomes strict for $\alpha \geq 2$.
\end{proof}

\bibliographystyle{plain}
  
\end{document}